\documentclass{amsart}
\usepackage{amsmath, amssymb}

\input xy
\xyoption{all}

\newcommand{\spn}{\mathop{\mathrm{span}}}

\newcommand{\free}[1]{\underset{\scriptscriptstyle #1}{\displaystyle{\ast}}\,}

\title[graph invariants]{Graph theoretic invariants for operator algebras associated to topological dynamics}

\author{Benton L. Duncan}

\address{Department of Mathematics\\
North Dakota State University\\
Fargo, North Dakota\\
USA}

\email{benton.duncan@ndsu.edu}



\begin{document}

\theoremstyle{plain}
\newtheorem{thm}{Theorem}
\newtheorem{lem}{Lemma}
\newtheorem{prop}{Proposition}
\newtheorem{cor}{Corollary}

\theoremstyle{definition}
\newtheorem{dfn}{Definition}
\newtheorem*{construction}{Construction}
\newtheorem*{example}{Example}

\theoremstyle{remark}
\newtheorem*{conjecture}{Conjecture}
\newtheorem*{acknowledgement}{Acknowledgements}
\newtheorem{rmk}{Remark}

\begin{abstract}
We expand on some invariants used for classifying nonselfadjoint operator algebras.  Specifically to nonselfadjoint operator algebras which have a conditional expectation onto a commutative diagonal we construct an edge-colored directed graph which can be used as an operator algebra invariant.
\end{abstract}

\maketitle
One important aspect of the nonselfadjoint operator algebras of directed graphs is the fact that the directed graph is an isomorphism invariant of the algebra \cite{KatsoulisKribs, Solel}.  More recently the same approach has been used to classify certain nonselfadjoint operator algebras coming from dynamical systems, see for example \cite{DavidsonKatsoulis, DavidsonKatsoulis2, DavidsonRoydor}.  All of these constructions share similar properties.  They all rely on characterizing the maximal ideal space of the algebra, and then using two dimensional nest representations to represent generators for the algebra (the edges for graph algebras and the actions for the dynamical systems).

The success of this approach in certain classes of nonselfadjoint algebras suggests that the commonalities of these approaches may be exploited in other contexts.  We have considered an approach to this idea that can be used for nonselfadjoint algebras which have a conditional expectation onto their commutative diagonal.  This of course includes the classes studied above but also can be used in other contexts (specifically in the edge-colored directed graph algebras introduced in \cite{Duncanedge}).  These other contexts, however, force us to extend the invariants in a new directions.  In the following we consider invariants similar to those studied in the papers cited above.  However we consider both small and significant variants of these to construct from an operator algebra a directed graph which serves as an invariant for the algebra.

It is easy to see that the directed graph (the construction of which mimics that from \cite{KatsoulisKribs}) does not capture all of the information.  Specifically there are contexts where a topology may be considered on the edge set for the directed graph, or a coloring function can be applied to the edge set of the directed graph.  In what follows we considers some of these invariants and how they may be used as operator algebra invariants.

Specifically, we show that associated to a nonselfadjoint operator algebra is a directed graph which is an isomorphism invariant.  The vertex set has a natural topology but the edge set of the directed graph may be topologized in various useful ways.  We show that the set of topologizations of the edge set also form an isomorphism invariant.  Finally we show that a coloring function can be assigned to the edge set and show that the associated edge-colored directed graph is also an isomorphism invariant.  We conclude with an example where the two algebras are isomorphic as Banach algebras but the edge-colored directed graph differentiates the two algebras.

We have aimed to make this paper as self contained as possible, including proofs of some well-known results.  We also wish to emphasize that the motivation of this paper was the commonalities between a series of papers studying various nonselfadjoint operator algebras.  The similarities and differences in approach and results between this paper and those in \cite{DavidsonRoydor} should be noted.  In this paper we have attempted to abstract much of the concepts and ideas from these papers.

To set the scene we will consider abstract operator algebras and look to generalize and build upon a class of invariants for these algebras.  We wish to emphasize that we are focusing on completely isometric isomorphism invariants.  Some of these invariants have proved useful for more general notions of equivalence of algebras (for example see \cite{KatsoulisKribs} and \cite{DavidsonKatsoulis} for some of these invariants in the context of algebraic isomorphism).

We will start by considering the diagonal of an operator algebra, specifically $\Delta(A) = A \cap A^*$ which is the largest $C^*$-subalgebra inside the algebra $A$.

\begin{prop}
Let $A$ and $B$ be operator algebras which are completely isometrically isomorphic, then $\Delta(A)$ and $ \Delta(B)$ are isomorphic as $C^*$-algebras.
\end{prop}

\begin{proof}
Let $\pi: A \rightarrow B$ be a completely isometric isomorphism.  Then $\pi|_{\Delta(A)}: \Delta(A) \rightarrow B$ is a completely isometric isomorphism onto its range $\pi(\Delta(A))$..  Since $\Delta(A)$ is a $C^*$-algebra this implies that $ \pi_{\Delta(A)}$ is a $C^*$-algebra isomorphism onto $\pi(\Delta(A))$.  It follows that $ \pi(\Delta(A))$ is a $C^*$-subalgebra of $B$ and hence $ \pi(\Delta(A)) \subseteq \Delta(B)$. Similarly we have that $\pi^{-1}(\Delta(B)) \subseteq \Delta(A)$.  Since $ \pi( \pi^{-1})|_{\Delta(B)}$ is the identity on $\Delta(B)$ and $ \pi^{-1}( \pi)|_{\Delta(A)}$ is the identity on $\Delta(A)$ it follows that $ \pi_{\Delta(A)}$ is onto $\Delta(B)$.
\end{proof}

Our invariants will focus on directed graphs and for this purpose we begin by defining the vertex set of the graph associated to $A$.  Given an operator algebra we define $V(A)$ to be the set multiplicative linear functionals of $\Delta(A)$.  Of course when $\Delta(A)$ is commutative this is the maximal ideal space of $\Delta(A)$ which completely characterizes $\Delta(A)$.

\begin{prop}
Let $A$ and $B$ be operator algebras which are completely isometrically isomorphic via $\pi$, then there is an equivalence between $V(A)$ and $V(B)$ given by $ \rho \mapsto \rho \circ \pi^{-1}$.  Further this map is weak$*$-weak$*$ continuous.
\end{prop}

\begin{proof}
Let $ \rho $ be a multiplicative linear functional for $V(A)$, then $ \rho \circ \pi^{-1}$ is a multiplicative linear functional in $V(B)$.  Similarly, if $ \sigma$ is a multiplicative linear functional in $V(B)$ then $ \sigma \circ \pi$ is a multiplicative linear functional in $V(A)$.  Notice that $ (\rho \circ \pi^{-1}) \circ \pi = \rho $ and $ (\sigma \circ \pi) \circ \pi^{-1} = \sigma$ hence the map $ \rho \mapsto \rho \circ \pi^{-1}$ is one-to-one and onto.

If $ \rho_{\lambda}$ is a convergent net in $V(A)$, converging to $\rho$ in the weak-$*$ topology then $\rho_{\lambda}(a)$ converges to $\rho(a)$ for every $ a \in \Delta(A)$.   Now if $ b \in \Delta(B)$ then $ b = \pi(a)$ for some $a \in \Delta(A)$ so that $ \rho_{\lambda} \circ \pi^{-1}(b) = \rho_{\lambda}(a)$ which converges to $ \rho(a) = \rho \circ \pi^{-1}(b)$ and hence $ \rho_{\lambda} \circ \pi $ converges to $ \rho \circ \pi$ in the weak-$*$ topology in $V(B)$.
\end{proof}

In this way $V(A)$ is a locally compact Hausdorff space.  We will write $ v \in V(A)$ to denote a multiplicative linear functional for $\Delta(A)$.

In what follows we will make an additional hypothesis for the pair $(A, \Delta(A))$.  Specifically we will assume that there is a completely contractive $\varphi: A \rightarrow \Delta(A)$ which is the identity when restricted to $\Delta(A)$, we will denote the existence of such a map by the triple $(A, \Delta(A), \varphi_A)$.  Notice that the map $ \varphi_A$ is unique.

\begin{prop}
Let $A$ and $B$ be operator algebras which are completely isometrically isomorphic via $\pi$ and assume the existence of the triple $(A, \Delta(A), \varphi_A)$, then the map $ \varphi_{\pi(A)} = \pi \circ \varphi_A \circ \pi^{-1}$ gives rise to the triple $(B, \Delta(B), \varphi_{\pi(A)})$.
\end{prop}

\begin{proof}
By definition $ \varphi_{\pi(A)}$ is completely contractive.  Further if $ b \in \Delta(B)$ then $ \pi^{-1}(b) \in \Delta(A)$ so that $ \varphi_A(\pi^{-1}(b)) = \pi^{-1}(b)$ so that $ \varphi_{\pi(A)}(b) = b$.
\end{proof}

We will then say that the triples $(A, \Delta(A), \varphi_A)$ and $ (B, \Delta(B), \varphi_B)$ are completely isometrically isomorphic via $\pi$ if there is a completely isometric isomorphism $ \pi: A \rightarrow B$.

We now consider the edge space for $(A, \Delta(A), \varphi_A)$.  We start by fixing two vertices $ v,w \in V(A)$.  We will then consider completely contractive representations $T: A \rightarrow T_2$ of the form
\[ T(a) = \begin{bmatrix} v \circ \varphi_A(a) & t(a) \\ 0 & v \circ \varphi_A(a) \end{bmatrix} \]
where $t$ is a nonzero linear map.  We call the collection of such representations $T_{w,v}$.  We now consider the subalgebra $K_{w,v} = \cap \{ \ker T: T \in T_{w,v} \}$.  For any pair, $v$ and $w$ we let $E_{w,v} := (\ker \varphi_A) / K_{w,v}$ and  $k_{w,v} = \dim E_{w,v}$.  If we fix a basis for $E_{w,v}$ then we have one edge with source $v$ and range $w$ in the graph for each basis element of $E_{w,v}$.  Then $E(A)$ will be the union of all such edges.

Notice in the preceding construction that $ \Delta(A) \subseteq \ker t$.

\begin{prop}
Let $(A, \Delta(A), \varphi_A)$ and $(B, \Delta(B), \varphi_B)$ be operator algebras which are completely isometrically isomorphic via $\pi$, then for any pair $v,w \in V(A)$ there is an isomorphism between $E_{w,v}$ and $E_{\pi(w), \pi(v)}$ as vector spaces over $\mathbb{C}$.
\end{prop}

\begin{proof}
Notice that if $ T: A \rightarrow T_2$ connecting $w$ and $v$ is given by then $ T \circ \rho^{-1}: B \rightarrow T_2$ connects $\pi(w)$ and $\pi(v)$.  Since $T$ and $ \rho^{-1}$ are linear the map $T \mapsto T \circ \rho^{-1})$ is a linear map between $E_{w,v}$ and $ E_{\pi(w). \pi(v)}$.  Also $T \circ \rho^{-1} \circ \rho = T$ we have that this map is one-to-one.  Surjection follows by constructing the reverse map.
\end{proof}

The corollary is now immediate.

\begin{cor}
Let $(A, \Delta(A), \varphi_A)$ and $(B, \Delta(B), \varphi_B)$ be operator algebras which are completely isometrically isomorphic via $\pi$ then there is an equivalence between $E(A)$ and $E(B)$.
\end{cor}

Unfortunately the edge sets need not necessarily have a direct correlation to elements of the algebra.  To see this consider the algebra $A_{\infty}$ which can be viewed as the graph algebra corresponding to a single vertex and a countable number of loops based at the vertex (where each edge coincides with a generator of the algebra).  However in our construction $E(A)$ will be uncountable.  This also presents us a distinction between our work and that of \cite{KatsoulisKribs}.  One could (if one is restricting to separable operator algebras) presumably just change the number of edges whenever an uncountable number of edges arise between a pair of vertices to a countable set (as was implicit in \cite{KatsoulisKribs}) without great affect.  However since we will need to work with the vector space $E_{w,v}$ this additional structure would assume some canonical choice of algebraic generators, which in the abstract involves some choice which may not be invariant under isomorphism.  However to further our analysis we will need to focus on those cases where such a choice is possible.

Now fix a basis for $E_{w,v}$, call it $ \{ e_{\lambda} \}$ and define the map $t_{\lambda}: E_{w,v} \rightarrow \mathbb{C}$ by
\[ t_{\lambda} (e_{\mu}) =  \begin{cases} 1 & \mu = \lambda \\ 0 & \mu \neq \lambda \end{cases} \]
and extending by linearity.  Now if we let $q: A \rightarrow ( \ker \varphi_A)/K_{w,v}$ be the map which is zero on $\Delta(A)$ and is the canonical quotient on $ \ker \varphi_A$ then we can consider the map
\[ T_{\lambda}(a) = \begin{bmatrix} v( \varphi_A(a)) & t_{\lambda} (q(a)) \\ 0 & w( \varphi_A(a)) \end{bmatrix} \]
which goes from $A$ to $T_2$.  The reader should notice that most of the time this map need not be a homomorphism, let alone completely contractive.  We will say that $A$ is $v$-$w$ free if there is a choice of basis for $E_{w,v}$ such that the associated maps $T_{\lambda}$ are completely contractive homomorphisms.

\begin{prop}
Let $(A, \Delta(A), \varphi_A)$ and $(B, \Delta(B), \varphi_B)$ be operator algebras which are completely isometrically isomorphic via $\pi$ and $v, w \in V(A)$.  If $A$ is $v$-$w$ free then $B$ is $ \pi(v)$-$\pi(w)$ free.
\end{prop}

\begin{proof}
Since $E_{w,v}$ and $E_{\pi(w),\pi(v)}$ are isomorphic as vector spaces over $ \mathbb{C}$ this is immediate.
\end{proof}

We say that an algebra $(A, \Delta(A), \varphi_A)$ is edge free if it is $v$-$w$ free for every pair $v, w \in V(A)$.  Notice that if $E_{w,v}$ is finite dimensional then it is straightforward to check that $A$ is $v$-$w$ free and hence if $(A, \Delta(A), \varphi_A)$ is vertex-pair finite (meaning between any two vertices there are at most finitely many edges between the vertices) then  $(A, \Delta(A), \varphi_A)$ is edge free.

Notice that the choice of basis is not unique since one could multiply a basis element by a scalar of modulus one and get another basis with the same property.  We will however, when $(A, \Delta(A), \varphi_A)$ is edge free, fix a basis for each $E_{w,v}$ for which the maps $\{ T_{\lambda} \}$ are completely contractive homomorphisms.  Henceforth we will associate to each basis element an edge in the graph with the range and source maps being defined in the natural way.  Specifically for the basis $ \{ e_i \}$ for $E_{w,v}$ for which the maps $T_i$ are completely contractive, we define $r(e_i) = w, s(e_i) = v$.

We now consider paths in the graph.  Specifically we let $(V(A), E(A), r,s)$ be the graph for $(A, \Delta(A), \varphi_A)$ and we say that a path $ w = e_ne_{n-1} \cdots e_1$ is admissible for $A$ if there is a completely contractive representation $ \pi_w: A \rightarrow T_{n+1}$ such that:
\[ \pi_w(a) = \begin{bmatrix} r(e_n)( \varphi_A(a)) & t_n(a) & * & * & \cdots \\ 0 & r(e_{n-1})(\varphi_A(a)) & t_{n-1}(a) & * & \cdots  \\ \vdots & \ddots & \ddots & \ddots & \vdots \\ 0 & \cdots  & 0 & r(e_1)( \varphi_A(a)) & t_1(a) \\ 0 & \cdots & \cdots  & 0 & s(e_1)( \varphi_A(a)) \end{bmatrix} \]
where $*$ indicates a potentially nonzero entry.

\begin{prop}
Let $(A, \Delta(A), \varphi_A)$ and $(B, \Delta(B), \varphi_B)$ be edge free operator algebras which are completely isometrically isomorphic via $ \pi$ then there is an equivalence between admissible paths for $A$ and admissible paths for $B$.
\end{prop}

\begin{proof}
Begin with the admissible path $e_ne_{n-1} \cdots e_1$ in $A$ and consider the path $\pi(e_n)\pi(e_{n-1}) \cdots \pi(e_1)$ in $B$.  We will show that this path is admissible.  Equivalence of admissible paths will then follow.  Since $e_ne_{n-1} \cdots e_1$ is an admissible path for $A$ there exists a representation $\tau: A \rightarrow T_{n+1}$ given by
\[ \tau(a) = \begin{bmatrix} r(e_n) & e_1(a) & e_2e_1(a) & \cdots & e_n e_{n-1} \cdots e_1(a) \\ 0 & r(e_{n-1}) & e_2(a) & \cdots & e_{n-1}e_{n-2} \cdots e_1(a) \\ \vdots & \cdots & \ddots & \ddots & \vdots \\ 0 & \cdots & 0 & 0 & s(e_1) \end{bmatrix} \]
where $e_ne_{n-1} \cdots e_1$ is non zero.  Notice then that $\tau \circ \pi: B \rightarrow T_{n+1}$ gives a representation that forces $ \pi(e_n) \pi(e_{n_1}) \cdots \pi(e_1)$ to be admissible.
\end{proof}

We now consider topological graphs that are related to our algebra $(A, \Delta(A), \varphi_A)$ and the associated graph $(V(A),E(A), r, s)$.  To motivate this we consider the example of a multivariate tensor algebra:

\begin{example}
Specifically if $ \alpha = \{\alpha_i \}$ is a finite collection of $*$-automorphisms of $C(X)$ then we consider the algebra $A(X, \alpha)$ as defined in \cite{DavidsonKatsoulis}.  The graph associated to this consists of $X$ as the vertex set and to each elements of $X$ there is an edge $e_{i,x}$ such that $s(e_{i,x}) = x$ and $r(e_{i,x}) = \alpha_i(x)$.  However since this is a semicrossed product we have additional information.  Specifically we know that for fixed $i$ the two dimensional nest representations corresponding to the $e_{i,x}$ are all the range of the same isometry $S_i$ under the obvious mapping.  There is then some reason to connect the distinct $\{ e_{i,x} \}_{x \in X}$ as a single edge.
\end{example}

We consider a partition of the edge set $E(A) = \cup_{\lambda \in \Lambda} E_{\lambda}$ such that if $ e, f$ are in the same component then $s(e) \neq s(f)$.  Now for each $ \lambda$ we let $V_{\lambda} = \{ v \in V(A): v = s(e) \mbox{ for some } e \in E_{\lambda}$.  We let $V_{\lambda}$ inherit the weak-$*$ topology from $V(A)$ and we consider the subspace of $\Lambda \times V$ (with the product topology) given by $ \{ (\lambda, v): v \in V_{\lambda} \}$, call it $E_{\Lambda}$.  Next we define $s(\lambda, v) = v$ and $r(\lambda, v) = r(e)$ where $e \in E_{\lambda}$ and $s(e) = v$.  Notice that this is well defined since no two elements in the same component $E_{\lambda}$ share the same source.

As one can imagine it is unlikely given an arbitrary partition of the edge set that $(V(A), E_{\lambda}, r, s)$ is a topological graph.  We say that a partition is topological if we do in fact have a topological graph.  However, as the next example illustrates there may be multiple choices of partition which give rise to topological graphs.

\begin{example}
Consider the space $X = \{0,1, 2 \}$ and the homeomorphism $\alpha(x) = x$ and consider $A = C(X) \rtimes_{\alpha} \mathbb{Z}^+$.  Then the directed graph associated to $A$ is the graph
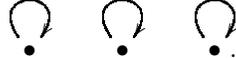
\[ \xymatrix{ {\bullet} \ar@(ul,ur) & {\bullet} \ar@(ul,ur) & {\bullet} \ar@(ul,ur) }. \]
Any partition of the edge set will give rise to a topological graph.
\end{example}

It is not hard to see that if we consider the trivial partition of $E(A)$ into singleton sets we get a topological partition.  But this may not be suitable for the question at hand.  For example, in \cite{DavidsonRoydor} a canonical partition is given for the algebras they consider, which coincides with the trivial partition only when the edge set has the discrete topology.  We should point out that the construction of Davidson and Roydor need not give rise to a reasonable partition for algebras which are not tensor algebras as the next example illustrates.

\begin{example}
Let $X= [0,1]$ and $ \alpha_1(x) = x$ and $ \alpha_2(x) = x$.  Then the semicrossed product $C(X) \rtimes_{\alpha_1, \alpha_2} \mathbb{F}_2^+$ gives rise to a directed graph of the form with vertex set homeomorphic to $[0,1]$ and at each point in $[0,1]$ there are two edges with range and source given by that point.

Now the equivalence relation given in Davidson-Roydor will give rise to the trivial partition where every edge is equivalent to every other edge.  Of course since we know where the edges come from we would most likely (again depending on context) wish to consider the partition that separates those edges corresponding to $\alpha_1$ from those corresponding to $\alpha_2$.
\end{example}

This is not surprising since there is additional structure Davidson and Roydor exploit in their context, specifically for their operator algebras the range of the partial isometry associated to two edges must be disjoint.  Of course our example, specifically in the case of semicrossed products illustrates a barrier to extending the results of \cite{DavidsonRoydor}.  Barring further evidence for the use of a specific partition we take the view that multiple partitions may be equally useful depending on context.

\begin{prop}
Let $(A, \Delta(A), \varphi_A)$ and $(B, \Delta(B), \varphi_B)$ be operator algebras which are completely isometrically isomorphic via $ \pi$ then there is an equivalence between topological partitions of $E(A)$ and topological partitions of $E(B)$.
\end{prop}

\begin{proof}
Fix a topological partition for $E(A)$ and define a relation on $E(B)$ by $e \sim f$ if $ e= \pi(e')$, $ f = \pi(f')$, and $e' \sim' f'$ in $E(A)$, where $\sim'$ is the equivalence relation on $E_{\Lambda}(A)$ giving rise to the topological partition.  That $ \sim$ is an equivalence relation on $E(B)$ comes from the fact that $\pi$ induces an equivalence between $E(A)$ and $E(B)$.  Now since $\sim'$ induces a topological partition on $E(A)$ we know that the range map $r: E(A) / \sim' \rightarrow V(A)$ is continuous and proper and the source map $ s: E(A) / \sim' \rightarrow V(A)$ is a local homeomorphism.  But since the map $\pi$ induces a homeomorphism between $V(A)$ and $V(B)$ we have that $ [\pi(r(\pi^{-1}(e)))] = r([e])$ is continuous and proper and $ [\pi(s( \pi^{-1}(e)))] = s([e])$ is a local homeomorphism and hence the partition of $E(B)$ is topological.
\end{proof}

Given two topological partitions $\Lambda_1$ and $ \Lambda_2$ of $E(A)$ we say that $\Lambda_1$ is less than or equal to $\Lambda_2$ if whenever $ e$ is equivalent to $f$ in $\Lambda_2$ then $e$ is equivalent to $f$ in $\Lambda_1$.  Notice that the trivial partition is always the smallest partition, and the discrete partition is always the largest partition under this partial ordering.

\begin{prop}
Let $(A, \Delta(A), \varphi_A)$ and $(B, \Delta(B), \varphi_B)$ be operator algebras which are completely isometrically isomorphic via $\pi$ then the equivalence relations between topological partitions of $E(A)$ and topological partitions of $E(B)$ preserve ordering.
\end{prop}

\begin{proof}
Given topological partitions $\Pi_1$ and $ \Pi_2$ of $E(A)$ with $ \Pi_1 \leq \Pi_2$ we wish to show that $ \pi( \Pi_1) \leq \pi( \Pi_2)$.  This is straightforward since the equivalences for $E(A)$ will transfer via $\pi$ to equivalences on $E(B)$ which preserve the ordering on partitions.
\end{proof}

We will now fix a topological partition for an edge-free operator algebra, and denote it's topological graph as $(V(A), E_{\Lambda}, r, s)$.  We will also assume that $\Lambda$ is a countable set.  In an effort to differentiate semicrossed products from tensor algebras we consider the extension of an idea from \cite{Duncanedge}.  We say that a function $f: \Lambda \rightarrow \mathbb{N}$ is an edge-coloring (\cite{AraGoodearl} call this a separation) of $ (V(A), E_{\lambda}, r, s)$ if
$f( \lambda_1) \neq f(\lambda_2)$ when there exist edges $e_1 \in E_{\lambda_1}$ and $e_2 \in E_{\lambda_2}$ such that $r(e_1) = r(e_2)$ and there is a completely contractive representation $\pi: A \rightarrow T_3$ given by
\[ \pi(a) = \begin{bmatrix} r(e_1) & t_1(a) & t_2(a) \\ 0 & s(e_1) & 0 \\ 0 & 0 & s(e_2) \end{bmatrix}.\]

\begin{prop}
Let $(A, \Delta(A), \varphi_A)$ and $(B, \Delta(B), \varphi_B)$ be edge-free operator algebras which are completely isometrically isomorphic via $\pi$ and a fixed topological partition of $E(A)$, call it $ \mathcal{E}$ then there is a equivalence between coloring functions on $\mathcal{E}$ and coloring functions on the $\pi(\mathcal{E})$.
\end{prop}

\begin{proof}
Let $f$ be a coloring function on $\mathcal{E}$ then for $ e \in \pi(\mathcal{E})$ we define $\pi(f)(e) = (f \circ \pi)(e)$.  Then $\pi(f)$ is a function from $\pi(\mathcal{E})$ into $ \mathbb{N}$.  That the map $ f \mapsto \pi(f)$ is a bijection is straightforward.  We wish to verify that $ \pi(f)$ is a coloring function for $\pi( \mathcal{E})$.  To see this assume that $[e] \neq [f] \in \pi(\mathcal{E})$ are differently colored.  In particular this means that there exist edges $ e \in [e]$ and $f \in [f]$ with $r(e) = r(f)$ and a completely contractive representation $\tau: B \rightarrow T_3$ given by
\[ \tau(b) = \begin{bmatrix} r(e) & e(b) & f(b) \\ 0 & s(e) & 0 \\ 0 & 0 & s(f) \end{bmatrix} \] for all $b \in B$.
Then notice that $ \tau \circ \pi$ is a completely contractive representation of $A$ such that
\[ \tau \circ \pi (a) = \begin{bmatrix} \pi^{-1}(r(e)) & \pi^{-1}(e)(a) & \pi^{-1}(f)(b) \\ 0 & \pi^{-1}(s(e)) & 0 \\ 0 & 0 & \pi^{-1}(s(f)) \end{bmatrix} \] for all $ a \in A$.  In particular this means that $ [ \pi^{-1}(e)]$ and $ [\pi^{-1}(f)]$ are differently colored in $ \mathcal{E}$.  Hence the map $ \pi(f)$ is a coloring function for $\pi(\mathcal{E}))$.
\end{proof}

We say that a coloring function $f$ is less than or equal to $g$ if there exists a permutation of $ \sigma: \mathbb{N} \rightarrow \mathbb{N}$ such that $f(n) \leq g(\sigma(n))$ for all $n \in \mathbb{N}$.

\begin{prop}
Let $(A, \Delta(A), \varphi_A)$ and $(B, \Delta(B), \varphi_B)$ be operator algebras which are completely isometrically isomorphic via $\pi$ with topological partition $ \mathcal{E}$ for $E(A)$ and $\pi(\mathcal{E})$ for $E(B)$.  If $ f_1 \leq f_2$ are coloring functions then $ \pi(f_1) \leq \pi(f_2)$.
\end{prop}

\begin{proof}
This is a straightforward application of the definition of the ordering on coloring functions and the relationship between a coloring function $f$ for $\mathcal{E}$ and $ \pi(f)$.
\end{proof}

Notice that a coloring function is maximal if no two edges with common range have the same coloring.  Of course there will be many inequivalent maximal coloring functions and potentially many inequivalent minimal coloring functions.

\begin{cor}
Let $(A, \Delta(A), \varphi_A)$ and $(B, \Delta(B), \varphi_B)$ be operator algebras which are completely isometrically isomorphic via $ \pi$ with topological partitions $ \mathcal{E}$ for $E(A)$ and $ \pi( \mathcal{E})$ for $E(B)$.  If $f$ is a maximal coloring function for $\mathcal{E}$ then there $\pi(f)$ is a maximal coloring function for $\pi(\mathcal{E})$.  Similarly there is an equivalence between minimal coloring functions for $\mathcal{E}$ and $ \pi(\mathcal{E})$.
\end{cor}

Based on what we have so far we have the following theorem.

\begin{thm}
Let $A$ be an operator algebra then the following are completely isometric isomorphism invariants for $A$
\begin{enumerate}
\item The graph $G(A)$.
\item Edge freeness of $A$.
\item Admissible paths in the graph $G(A)$ when $A$ is edge free.
\item The weak-$*$ topology on $V(A)$.
\item The partially ordered set of topological partitions of $E(A)$.
\item The partially ordered set of coloring functions for a fixed topological partition of $E(A)$.
\item The collection of edge colored topological graphs compatible with $G(A)$ when $A$ is edge free.
\end{enumerate}
\end{thm}

Notice that this can help distinguish between semicrossed products and tensor algebras for multivariable dynamics by way of the following result.

\begin{prop} A semicrossed product algebra is completely isometrically isomorphic to a tensor algebra if and only if the graph for the semicrossed product is $1$-colorable.
\end{prop}

\begin{proof} For one direction of this notice that the graph for tensor algebras is always $1$-colorable.  For the converse notice that if the graph associated to the semicrossed product $C(X) \rtimes_{\alpha} \mathbb{F}_n^+$ is $1$-colorable then the range of each of the edges must be disjoint.  Specifically this tells us that $S_iS_i^* S_jS_j^* = 0$ for each of the generating isometries $S_i$.  It follows that $ \sum S_iS_i^* \leq 1$ which is the universal property for the tensor algebra.  Hence there is a completely contractive map from the tensor algebra onto the semicrossed product.  The reverse map is universal and hence the two coincide.
\end{proof}

Notice that in this result completely isometrically isomorphic is the best we can hope for.  Specifically there are semicrossed products and tensor algebras which are isomorphic as Banach algebras but not as operator algebras as the next example illustrates.

\begin{example} We consider the the four point set $ \{ 1, 2, 3, 4 \}$ and the pair of maps $ f(1) = 2, f(2) = f(3) = f(4) =3$ and $g(1) = 2, g(2) = g(3) = g(4) = 4$.  Notice that the graph for $A(X, \{ f, g \})$ and the graph for $C(X) \rtimes_{\{ f, g \}} \mathbb{F}^+_2$ are both equal to 
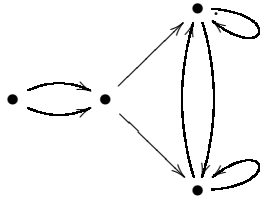
\[ \xymatrix{ & & {\bullet} \ar@(r,dr) \ar@/^/[dd] \\ {\bullet} \ar@/^/[r] \ar@/_/[r] & {\bullet} \ar[ur] \ar[dr] & \\ & & {\bullet} \ar@(r,ur) \ar@/^/[uu]}.\]
In fact if we denote this graph by $G$ the $A(X, \{ f, g \})$ is the directed graph algebra $A(G)$, and $C(X) \rtimes_{\{ f, g \}} \mathbb{F}^+_2$ is the edge colored directed graph algebra $A(G, \chi)$ where $\chi(e_i) = \begin{cases} 1 \mbox{ if } i \mbox{ is odd} \\ 2 \mbox{ if } i \mbox{ is even} \end{cases}.$

Using \cite{Duncanedge} we can see that $A(G) = \left( A(G') \oplus \mathbb{C} \right) \free{\mathbb{C}^4} \left( A' \oplus \mathbb{C}^2 \right)$ where $G'$ is the graph 
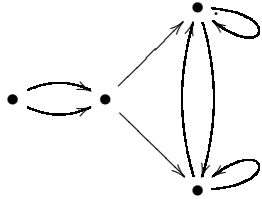
\[ \xymatrix{ & & {\bullet} \ar@(r,dr) \ar@/^/[dd] \\ {\bullet} \ar@/^/[r] \ar@/_/[r] & {\bullet} \ar[ur] \ar[dr] & \\ & & {\bullet} \ar@(r,ur) \ar@/^/[uu]}.\] and $A$ is the algebra \[ \left\{ \begin{bmatrix} \lambda_1 & 0 & 0 \\ \lambda_2 & \lambda_3 & 0 \\ \lambda_4 & 0 & \lambda_3\end{bmatrix} : \lambda_i \in \mathbb{C} \right\}. \]

Similarly $A(G, \chi)$ is given as a free product of the form $ \left( A(G') \oplus \mathbb{C} \right) \free{\mathbb{C}^4} \left( T^2 \oplus \mathbb{C}^2\right)$.  Identifying the subalgebras given by $A(G')$ and noticing that $P_2+P_3+P_4 X P_2+P_3+P_4 = X$ for all $X \in A(G')$ (this is because the equation holds for all of the generators of $A(G')$) we can write every element of $A(G, \chi)$ in the form $ \lambda + T_1 X + T_2Y + Z$ where $T_i$ is the partial isometry associated to the edges $e_1$ and $e_2$ in $G$ and $X, Y, $ and $Z$ are all in $A(G')$.  A similar construction for $A(G)$ gives us an algebraic isomorphism between $A(G, \chi)$ and $A(G)$.  But in this context algebraic isomorphism implies continuous isomorphism \cite{DavidsonKatsoulis} and hence the two algebras are isomorphic as Banach algebras.  However they are not completely isometrically isomorphic since the second algebra gives rise to an edge-colored directed graph which is not $1$-colorable.
\end{example}

\bibliographystyle{plain}

\end{document}